\documentclass[12pt,a4paper,reqno]{amsart}
\usepackage{fullpage}
\usepackage{verbatim}
\usepackage{hyperref}
\usepackage{pslatex}

\newcommand{\odip}[2]{o _{#1}\!\left(#2\right)\mathchoice{\!}{}{}{}}
\newcommand{\odi}[1]{\odip{}{#1}}
\newcommand{\Odip}[2]{\mathcal{O}_{#1}\left(#2\right)}
\newcommand{\Odi}[1]{\mathcal{O}\left(#1\right)}
\newcommand{\dx}{\mathrm{d}}
\newcommand{\e}{\mathrm{e}}
\newcommand{\ii}{\mathrm{i}}
\newcommand{\E}{\widetilde{\mathcal{E}}}
\newcommand{\C}{\mathcal{C}}
\newcommand{\eps}{\varepsilon}
\newcommand{\Stilde}{\widetilde{S}}

\newtheorem{Theorem}{Theorem}[section]
\newtheorem{Lemma}{Lemma}[section]

\allowdisplaybreaks

\title[Representations of an integer as a sum of prime powers]
  {On the average number of representations of an integer \\
  as a sum of like prime powers}

\author[M.~Cantarini, A.~Gambini, A.~Zaccagnini]
       {Marco Cantarini, Alessandro Gambini, Alessandro Zaccagnini}

\date{\today}


\begin{document}

\begin{abstract}
We investigate the average number of representations of a positive
integer as the sum of $k + 1$ perfect $k$-th powers of primes.
We extend recent results of Languasco and the last Author, which
dealt with the case $k = 2$ \cite{LanguascoZ2017e} and
$k = 3$ \cite{LanguascoZ2017c} respectively.
We use the same technique to study the corresponding problem for sums
of just $k$ perfect $k$-th powers of primes.
\end{abstract}

\maketitle

\section{Introduction}

The problem of representing a large integer $n$, satisfying suitable
congruence conditions, as a sum of a prescribed number of powers
of primes, say $n = p_1^{k_1} + \cdots + p_s^{k_s}$, is classical.
Here $k_1$, \dots, $k_s$ denote fixed positive integers.
This class of problems includes both the binary and ternary Goldbach
problem, and Hua's problem.
If the \emph{density} $\rho = k_1^{-1} + \cdots + k_s^{-1}$ is large
and $s \ge 3$, it is often possible to give an asymptotic formula for
the number of different representations the integer $n$ has.
When the density $\rho$ is comparatively small, the individual problem
is usually intractable and it is reasonable to turn to the easier task
of studying the average number of representations, if possible
considering only integers $n$ belonging to a \emph{short} interval
$[N, N + H]$, say, where $H \ge 1$ is as small as possible.

Here we deal with the average number of representations of a positive
integer $n$ as the sum of $k + 1$ perfect $k$-th powers of prime
numbers.
The case $k = 2$ (actually, to be exact a slightly more general
problem) has been studied in \cite{LanguascoZ2017e}, while the case
$k = 3$ has been studied in \cite{LanguascoZ2017c}.
Here we give a uniform and simpler proof which is valid for general
$k \ge 2$.
Let
\begin{equation}
\label{Rk-def}
  R_k(n)
  =
  \sum_{n = m_1^k + \cdots + m_{k + 1}^k}
    \Lambda(m_1) \cdots \Lambda(m_{k + 1}),
\end{equation}
where $\Lambda$ is the von Mangoldt function, that is,
$\Lambda(p^m) = \log(p)$ if $p$ is a prime number and $m$ is a
positive integer, and $\Lambda(n) = 0$ for all other integers.

\begin{Theorem}
\label{CGZ-uncond}
Let $k \ge 2$ be a fixed integer.
For every $\eps > 0$ there exists a constant $C = C(\eps) > 0$,
independent of $k$, such that
\[
  \sum_{n = N + 1}^{N + H} R_k(n)
  =
  \Gamma\Bigl(1 + \frac1k \Bigr)^k H N^{1/k}
  +
  \Odip{k}{H N^{1 / k}
    \exp\Bigl\{ -C \Bigl(\frac{\log N}{\log\log N}\Bigr)^{1/3} \Bigr\}}
\]
as $N \to +\infty$, uniformly for
$N^{1 - 5 / (6k) + \eps} < H < N^{1 - \eps}$, where
$\Gamma$ is the Euler Gamma-function.
\end{Theorem}

\begin{Theorem}
\label{CGZ-cond}
Let $k \ge 2$ be a fixed integer and assume that the Riemann
Hypothesis (RH) holds.
Then
\[
  \sum_{n = N + 1}^{N + H} R_k(n)
  =
  \Gamma\Bigl(1 + \frac1k \Bigr)^k H N^{1/k}
  +
  \Odip{k}{\Phi_k(N, H)},
\]
uniformly for $H = \infty(N^{1 - 1 / k} L^3)$ with $H = \odi{N}$,
where $f = \infty(g)$ means $g = \odi{f}$, $L = \log N$ and
$\Phi_k(N, H) = N L^3 + H^2 N^{1/k - 1} + H^{1 / 2} N^{1 / 2 + 1 / (2 k)} L
  +
  H N^{1 / (2 k)} L^{3 / 2}$.
\end{Theorem}

The density of this problem is $1 + 1 / k$.
Theorems~\ref{CGZ-uncond} and~\ref{CGZ-cond} contain as special cases
the results in \cite{LanguascoZ2017e} and \cite{LanguascoZ2017c}.
The limitation in Theorem~\ref{CGZ-uncond} is due to the corresponding
one in Lemma~\ref{LP-Lemma-gen}, whereas the limitation in
Theorem~\ref{CGZ-cond} is the expected one.

The main new ingredient is the use of the elementary identity
\begin{equation}
\label{identity}
  x^{k + 1} - y^{k + 1}
  =
  (x - y)^2
  \sum_{j = 1}^k j x^{k - j} y^{j - 1}
  +
  (k + 1) (x - y) y^k,
\end{equation}
which is valid for integral $k \ge 1$.
We will show in \S\ref{proof-uncond} below that this identity can be
used quite effectively to circumvent the need for sharp bounds for
exponential sums.
In fact, we manage to obtain the ``expected'' limitations (in view of
Lemma~\ref{LP-Lemma-gen}) for the length $H$ of the ``short interval,''
both in the unconditional and in the conditional case.

We remark that the use of identity \eqref{identity} allows us to treat
also the average value of the number of representations of an integer
as a sum of just $k$ perfect $k$-th powers of prime numbers, which is
a problem of density $1$.
Let
\begin{equation}
\label{R'k-def}
  R'_k(n)
  =
  \sum_{n = m_1^k + \cdots + m_k^k}
    \Lambda(m_1) \cdots \Lambda(m_k).
\end{equation}
We can prove the following results, which will also appear in a
forthcoming paper by A.~Languasco with a proof along the lines of
\cite{LanguascoZ2017c}.

\begin{Theorem}
\label{CGZ'-uncond}
Let $k \ge 2$ be a fixed integer.
For every $\eps > 0$ there exists a constant $C = C(\eps) > 0$,
independent of $k$, such that
\[
  \sum_{n = N + 1}^{N + H} R'_k(n)
  =
  \Gamma\Bigl(1 + \frac1k \Bigr)^k H
  +
  \Odip{k}{
    H \exp\Bigl\{ -C \Bigl(\frac{\log N}{\log\log N}\Bigr)^{1/3} \Bigr\}}
\]
as $N \to +\infty$, uniformly for
$N^{1 - 5 / (6k) + \eps} < H < N^{1 - \eps}$.
\end{Theorem}

\begin{Theorem}
\label{CGZ'-cond}
Let $k \ge 2$ be a fixed integer and assume that the Riemann
Hypothesis holds.
Then
\[
  \sum_{n = N + 1}^{N + H} R'_k(n)
  =
  \Gamma\Bigl(1 + \frac1k \Bigr)^k H
  +
  \Odip{k}{\Phi'_k(N, H)},
\]
uniformly for $H = \infty(N^{1 - 1 / k} L^3)$ with $H = \odi{N}$,
where
$\Phi'_k(N, H) = N^{1 - 1 / k} L^3 + H^2 N^{- 1} +
  H^{1 / 2} N^{1 / 2 - 1 / (2 k)} L + H^{1 - 1/k}  N^{1 / (2 k)} L^{3 / 2}$.
\end{Theorem}

We do not give the full detailed proofs of Theorems~\ref{CGZ'-uncond}
and \ref{CGZ'-cond}, but just a short summary in \S\ref{proof-density-1}. 
The starting point is the use of identity~\eqref{identity} with $k$ in
place of $k + 1$.
The case $k = 2$ of both results is proved in Languasco \& Zaccagnini
\cite{LanguascoZ2016c}.

\section{Definitions and preparation for the proofs}

For real $\alpha$ we write $\e(\alpha) = \e^{2 \pi \ii \alpha}$.
We take $N$ as a large positive integer, and write $L = \log N$ for
brevity.
Let $z = 1 / N - 2 \pi \ii \alpha$ and
\begin{equation}
\label{Stilde-def}
  \Stilde_k(\alpha)
  =
  \sum_{n \ge 1} \Lambda(n) \e^{-n^k / N} \e(n^k \alpha)
  =
  \sum_{n \ge 1} \Lambda(n) \e^{- n^k z}.
\end{equation}
Thus, recalling definition \eqref{Rk-def} and using \eqref{Stilde-def},
for all $n \ge 1$ we have
\begin{equation}
\label{basic-Rk}
  R_k(n)
  =
  \sum_{n_1^k + \cdots + n_{k+1}^k = n}
    \Lambda(n_1) \cdots \Lambda(n_{k+1})
  =
  \e^{n / N}
  \int_{-1/2}^{1/2}
    \Stilde_k(\alpha)^{k+1} \, \e(-n \alpha) \, \dx \alpha.
\end{equation}
It is clear from the above identity that we are only interested in the
range $\alpha \in [-1/2, 1/2]$.
We record here the basic inequality
\begin{equation}
\label{z-bound}
  \vert z \vert^{-1}
  \ll
  \min \{ N, \vert \alpha \vert^{-1} \}.
\end{equation}
We also need the following exponential sum over the ``short interval''
$[1, H]$
\[
  U(\alpha, H)
  =
  \sum_{m = 1}^H \e(m \alpha),
\]
where $H \le N$ is a large integer.
We recall the simple inequality
\begin{equation}
\label{U-bound}
  \vert U(\alpha, H) \vert
  \le
  \min \{ H, \vert \alpha \vert^{-1} \}.
\end{equation}
With these definitions in mind and recalling \eqref{basic-Rk}, we
remark that
\begin{equation}
\label{basic-identity}
  \sum_{n = N + 1}^{N + H}
    \e^{-n / N} R_k(n)
  =
  \int_{-1/2}^{1/2} \Stilde_k(\alpha)^{k+1} U(-\alpha, H) \e(-N \alpha)
    \, \dx \alpha,
\end{equation}
which is the starting point for our investigation.
The basic strategy is to replace $\Stilde_k(\alpha)$ by its expected
main term, which is $\Gamma(1 + 1 / k) / z^{1 / k}$, and estimating the
ensuing error term by means of identity \eqref{identity}.
We use trivial bounds for $\Stilde_k$ and $z^{-1 / k}$ and
Lemma~\ref{LP-Lemma-gen} to majorise each term in the sum on the
right-hand side of \eqref{identity}, and the Cauchy-Schwarz inequality
and the same Lemma again to majorise the summand on the far right.
Of course, we may use Lemma~\ref{LP-Lemma-gen} only in a restricted
range, and we need a different argument on the remaining part of the
integration interval.
This leads to some complications in details.
In the conditional case, we have no such limitations and our result
holds in the ``natural'' range for $H$.
The details are in \S\ref{proof-cond}.
In both cases, we achieve the proof by removing the extraneous
factor $\e^{- n / N}$ from the left-hand side of \eqref{basic-identity}.

\section{Lemmas}

It will shorten our formulae somewhat to write
$\gamma_k = \Gamma(1 + 1 / k)$.
For brevity, we also set
\[
  \E_k(\alpha)
  :=
  \Stilde_k(\alpha)
  -
  \frac{\gamma_k}{z^{1 / k}}
  \qquad\text{and}\qquad
  A(N; c)
  :=
  \exp\Bigl\{ c \Bigl(\frac{\log N}{\log\log N}\Bigr)^{1/3} \Bigr\},
\]
where $c$ is a real constant.

\begin{Lemma}[Lemma~3 of \cite{LanguascoZ2016a}, Lemma~1 of
\cite{LanguascoZ2016b}]
\label{LP-Lemma-gen}
Let $\eps$ be an arbitrarily small positive constant, $k \ge 1$ be
an integer, $N$ be a sufficiently large integer and $L = \log N$.
Then there exists a positive constant $c_1 = c_1(\eps)$, which does
not depend on $k$, such that
\[
  \int_{-\xi}^{\xi}
    \bigl\vert \E_k(\alpha) \bigr\vert^2 \, \dx \alpha
  \ll_k
  N^{2 / k - 1} A(N; - c_1)
\]
uniformly for $0 \le \xi < N^{ -1 + 5 / (6 k) - \eps}$.
Assuming the Riemann Hypothesis we have
\[
  \int_{-\xi}^{\xi} \,
    \bigl\vert \E_k(\alpha) \bigr\vert^2 \, \dx \alpha
  \ll_k
  N^{1 / k}\xi L^2
\]
uniformly for $0 \le \xi \le 1 / 2$.
\end{Lemma}

We remark that the proof of Lemma~\ref{LP-Lemma-gen} in
\cite{LanguascoZ2016a} contains oversights which are corrected in
\cite{LanguascoZ2017e}.
The next result is a variant of Lemma~4 of \cite{LanguascoZ2016a}: we
just follow the proof until the last step.
We need it to avoid dealing with the ``periphery'' of the major arc in
the unconditional case.

\begin{Lemma}
\label{Laplace-formula}
Let $N$ be a positive integer, $z = z(\alpha) = 1 / N - 2 \pi \ii \alpha$,
and $\mu > 0$.
Then, uniformly for $n \ge 1$ and $X > 0$ we have
\[
  \int_{-X}^X z^{-\mu} \e(-n \alpha) \, \dx \alpha
  =
  \e^{- n / N} \frac{n^{\mu - 1}}{\Gamma(\mu)}
  +
  \Odip{\mu}{\frac1{n X^{\mu}}}.
\]
\end{Lemma}

\begin{Lemma}
\label{Stilde-bound}
We have $\Stilde_k(\alpha) \ll_k N^{1 / k}$.
\end{Lemma}

\begin{proof}
It is a straightforward application of partial summation and a crude
form of the Prime Number Theorem.
In fact, recalling that the summatory function of the von Mangoldt
$\Lambda$-function $\psi(X)$ satisfies $\psi(X) \ll X$, we have
\begin{align*}
  \sum_{n \ge 1} \Lambda(n) \e^{- n^k / N}
  &=
  \lim_{X \to + \infty}
    \sum_{n \le X} \Lambda(n) \e^{- n^k / N} \\
  &=
  \lim_{X \to + \infty}
  \Bigl(
    \psi(X) \e^{- X^k / N}
    +
    \frac kN \int_0^X \psi(t) t^{k - 1} \e^{- t^k / N} \, \dx t
  \Bigr) \\
  &\ll_k
  N^{-1}
  \int_0^{+\infty} t^k \e^{- t^k / N} \, \dx t
  \ll_k
  N^{1 / k}
  \int_0^{+\infty} u^{1 / k} \e^{- u} \, \dx u,
\end{align*}
by a trivial change of variables.
\end{proof}

Our next tool is Lemma~6 of Languasco \& Zaccagnini
\cite{LanguascoZ2017c}: it is a consequence of Lemma~4 of
\cite{GambiniLZ2018}, which depends, essentially, on a result of
Robert \& Sargos \cite{RobertS2006}.

\begin{Lemma}
\label{fourth-power}
For $N$ a positive integer, $\tau > 0$ and for real $k > 1$
and real $\eps > 0$ we have
\[
  \int_{-\tau}^\tau \vert \Stilde_k(\alpha) \vert^4 \dx \alpha
  \ll_k
  (\tau N^{2/k} + N^{4/k-1}) N^\eps.
\]
\end{Lemma}

\begin{Lemma}
\label{fourth-power-avg}
For $k > 1$ and $N^{-c} < \tau \le N^{2 / k - 1}$ we have
\[
  \int_\tau^{1/2} \vert \Stilde_k(\alpha) \vert^4 \frac{\dx \alpha}\alpha
  \ll_k
  N^{4 / k - 1 + \eps} \tau^{-1}.
\]
\end{Lemma}

\begin{proof}
We just need a partial integration from Lemma~\ref{fourth-power}: let
\[
  F(\xi)
  =
  \int_0^\xi \vert \Stilde_k(\alpha) \vert^4 \dx \alpha
  \ll_k
  (\xi N^{2/k} + N^{4/k-1}) N^\eps.
\]
Now
\[
  \int_\tau^{1/2} \vert \Stilde_k(\alpha) \vert^4 \frac{\dx \alpha}\alpha
  =
  \Bigl[ \frac{F(\alpha)}{\alpha} \Bigr]_\tau^{1/2}
  +
  \int_\tau^{1/2} \frac{F(\alpha)}{\alpha^2} \, \dx \alpha
  \ll_k
  N^{4 / k - 1 + \eps} \tau^{-1},
\]
since $\tau \le N^{2 / k - 1}$.
\end{proof}

\begin{Lemma}
\label{mt-evaluation}
For $N \to +\infty$, $H \in [1, N]$ and a real number $\lambda$ we
have
\[
  \sum_{n = N + 1}^{N + H}
    \e^{- n / N} n^{\lambda}
  =
  \frac1{\e}
  H N^{\lambda}
  +
  \Odip{\lambda}{H^2 N^{\lambda - 1}}.
\]
\end{Lemma}

\begin{proof}
Using the approximation $\e^{-n / N} = \e^{-1} + \Odi{H N^{-1}}$
introduces an error $\Odip{\lambda}{H^2 N^{\lambda - 1}}$.
If $\lambda = 0$ there is nothing left to prove.
If $\lambda < 1$ with $\lambda \ne 0$ we have
\[
  n^{\lambda} - N^{\lambda}
  =
  \lambda
  \int_N^n t^{\lambda - 1} \, \dx t
  \ll_{\lambda}
  (n - N)
  N^{\lambda - 1}
  \le
  H N^{\lambda - 1}
\]
for all $n \in [N + 1, N + H]$, and we see that
\[
  \sum_{n = N + 1}^{N + H} n^{\lambda}
  =
  H N^{\lambda}
  +
  \Odip{\lambda}{H^2 N^{\lambda - 1}}.
\]
If $\lambda \ge 1$ the argument above proves that
$n^{\lambda} - N^{\lambda} \ll_{\lambda} H (N + H)^{\lambda - 1}
  \ll_{\lambda} H N^{\lambda - 1}$
since $H \le N$, and the conclusion follows.
\end{proof}

\section{Proof of Theorem \ref{CGZ-uncond}}
\label{proof-uncond}

We need to introduce another parameter $B = B(N)$, defined as
\begin{equation}
\label{def-B}
  B = N^{2 \eps}
\end{equation}
Ideally, we would like to take $B = 1$, but we are prevented from
doing this by the estimate in \S\ref{sub-I3}.
We let $\C = \C(B,H) = [-1/2, -B/H] \cup [B/H, 1/2]$.
Recalling \eqref{basic-identity} we write
\begin{align*}
  \sum_{n = N + 1}^{N + H}
    \e^{-n / N} R_k(n)
  &=
  \gamma_k^{k + 1}
  \int_{-B/H}^{B/H} \frac{U(-\alpha, H)}{z^{(k + 1) / k}}
    \, \e(-N \alpha) \, \dx \alpha \\
  &\qquad+
  \int_{-B/H}^{B/H}
    \Bigl(
      \Stilde_k(\alpha)^{k+1}
      -
      \frac{\gamma_k^{k + 1}}{z^{(k + 1) / k}}
    \Bigr)
    U(-\alpha, H) \e(-N \alpha) \, \dx \alpha \\
  &\qquad+
  \int_\C \Stilde_k(\alpha)^{k+1} U(-\alpha, H) \e(-N \alpha) \, \dx \alpha \\
  &=
  \gamma_k^{k + 1} I_1 + I_2 + I_3,
\end{align*}
say.
The first summand gives rise to the main term via
Lemma~\ref{Laplace-formula}, the second one is majorised by means of
identity \eqref{identity} and the $L^2$-estimate provided by
Lemma~\ref{LP-Lemma-gen}, and the last one is easy to bound using
Lemma~\ref{fourth-power-avg}.

\subsection{Evaluation of \texorpdfstring{$I_1$}{I1}}
\label{sub-I1}

It is a straightforward application of Lemma~\ref{Laplace-formula}:
here we exploit the flexibility of having variable endpoints instead
of the full unit interval.
We have
\begin{equation}
\label{prep-mt}
  I_1
  =
  \int_{-B/H}^{B/H} \frac{U(-\alpha, H)}{z^{(k + 1) / k}}
    \, \e(-N \alpha) \, \dx \alpha
  =
  \frac1{\gamma_k}
  \sum_{n = N + 1}^{N + H} \e^{-n / N} n^{1 / k}
  +
  \Odip{k}{\frac HN \Bigl( \frac HB \Bigr)^{(k + 1) / k}}.
\end{equation}
We evaluate the sum on the right-hand side of \eqref{prep-mt} by means
of Lemma~\ref{mt-evaluation} with $\lambda = 1 / k$.
Summing up, we have
\begin{equation}
\label{final-mt}
  I_1
  =
  \frac1{\e \gamma_k} H N^{1 / k}
  +
  \Odip{k}{H^2 N^{1 / k - 1} + \frac HN \Bigl( \frac HB \Bigr)^{(k + 1) / k}}.
\end{equation}
We can neglect the second summand in the error term since $H \le N$
and $B \ge 1$.

\begin{comment}
It is a straightforward application of Lemma~\ref{Laplace-formula},
the only minor difficulty being the fact that we have to deal with the
``restricted'' range $[-B / H, B / H]$ instead of the full unit
interval.
Using \eqref{z-bound} and \eqref{U-bound} and standard estimates we get
%
\begin{align}
\notag
  \int_{-B/H}^{B/H} \frac{U(-\alpha, H)}{z^{(k + 1) / k}}
    \, \e(-N \alpha) \, \dx \alpha
  &=
  \frac1{\gamma_k}
  \sum_{n = N + 1}^{N + H} \e^{-n / N} n^{1 / k}
  +
  \Odip{k}{\frac HN}
  +
  \Odi{
  \int_{B / H}^{1 / 2}
    \frac{\vert U(-\alpha, H) \vert}{\alpha^{1 + 1 / k}} \, \dx \alpha} \\
\notag
  &=
  \frac1{\gamma_k}
  \sum_{n = N + 1}^{N + H}
    \Bigl( \e^{-1} + \Odi{H N^{-1}} \Bigr) n^{1 / k}
  +
  \Odip{k}{\frac HN} \\
\notag
  &\qquad+
  \Odi{ \int_{B / H}^{1 / 2}
    \frac{\dx \alpha}{\alpha^{2 + 1 / k}}} \\
\notag
  &=
  \frac1{\e \gamma_k}
  \sum_{n = N + 1}^{N + H} n^{1 / k}
  +
  \Odip{k}{H^2 N^{1/k - 1} + H N^{-1} + (H / B)^{1 + 1/k}} \\
\label{eval-I1}
  &=
  \frac1{\e \gamma_k} H N^{1 / k}
  +
  \Odip{k}{N^{1/k} + H^2 N^{1/k - 1} + (H / B)^{1 + 1/k} },
\end{align}
%
since $H N^{-1}$ is smaller than $N^{1/k}$.
We remark that, in view of the future choice of $H$ due to the
limitations in Lemma~\ref{LP-Lemma-gen}, the error term is dominated
by $H^2 N^{1/k - 1} + (H / B)^{1 + 1/k}$.
\end{comment}

\subsection{Bound for \texorpdfstring{$I_2$}{I2}}
\label{sub-I2}

We let $x = x(\alpha) = \Stilde_k(\alpha)$ and
$y = y(\alpha) = \gamma_k z^{-1/k}$ and use \eqref{identity}.
We recall the bounds~\eqref{z-bound} and \eqref{U-bound}, and
Lemma~\ref{Stilde-bound}.
Using Lemma~\ref{LP-Lemma-gen} and the Cauchy-Schwarz inequality where
appropriate, we have
\begin{align}
\notag
  I_2
  &\ll_k
  H
  \sum_{j = 1}^k
    \int_{-B / H}^{B / H}
      \vert x - y \vert^2 \cdot \vert x \vert^{k - j} \cdot
      \vert y \vert^{j - 1} \, \dx \alpha
  +
  H
  \int_{-B / H}^{B / H}
    \vert x - y \vert \cdot \vert y \vert^k \, \dx \alpha \\
\notag
  &\ll_k
  H
  \sum_{j = 1}^k
    \max_{\alpha} \vert x \vert^{k - j} \cdot
    \max_{\alpha} \vert y \vert^{j - 1}
    \int_{-B / H}^{B / H} \vert x - y \vert^2 \, \dx \alpha \\
\notag
  &\qquad+
  H
  \Bigl(
  \int_{-B / H}^{B / H} \vert x - y \vert^2 \, \dx \alpha
  \int_{-B / H}^{B / H} \vert y \vert^{2 k} \, \dx \alpha
  \Bigr)^{1 / 2} \\
\notag
  &\ll_k
  H
  \sum_{j = 1}^k
    N^{(k - j) / k} \cdot N^{(j - 1) / k}
    N^{2 / k - 1} A(N; -c_1) \\
\notag
  &\qquad+
  H N^{1 / k - 1 / 2} A \Bigl(N; - \frac12 c_1\Bigr)
  \Bigl(
  \int_{-B / H}^{B / H} \frac{\dx \alpha}{\vert z \vert^2}
  \Bigr)^{1 / 2} \\
\label{bound-I2}
  &\ll_k
  H N^{1 / k} A \Bigl(N; - \frac12 c_1\Bigr),
\end{align}
where $c_1 = c_1(\eps) > 0$ is the constant provided by
Lemma~\ref{LP-Lemma-gen}, which we can use if
$B / H < N^{ -1 + 5 / (6 k) - \eps}$.
Recalling the choice in \eqref{def-B}, we see that we can take
\begin{equation}
\label{H-bound}
  H > N^{1 - 5 / (6 k) + 3 \eps}.
\end{equation}

\subsection{Bound for \texorpdfstring{$I_3$}{I3}}
\label{sub-I3}

For $k \ge 3$ we have
\begin{align}
\notag
  I_3
  =
  \int_\C \Stilde_k(\alpha)^{k+1} U(-\alpha, H) \e(-N \alpha) \, \dx \alpha
  &\ll_k
  \max_{\alpha \in [-1/2, 1/2]}
    \vert \Stilde_k(\alpha) \vert^{k - 3}
  \int_\C \vert \Stilde_k(\alpha) \vert^4 \, \frac{\dx \alpha}\alpha \\
\label{bound-I3}
  &\ll_k
  N^{(k - 3) / k} \cdot N^{4 / k - 1 + \eps} (B / H)^{-1}
  \ll_k
  N^{1 / k + \eps} H / B,
\end{align}
by Lemmas~\ref{Stilde-bound} and~\ref{fourth-power-avg}.
This is $\ll_k N^{1 / k} H A(N; -c_1 / 2)$, by our choice
in~\eqref{def-B}.
For $k = 2$ we can use a slightly different argument, based on
Lemma~5 of \cite{LanguascoZ2017e}.
We omit the details.

\subsection{Completion of the proof}
\label{final-Th1}

For simplicity, from now on we assume that $H \le N^{1 - \eps}$.
Summing up from \eqref{final-mt}, \eqref{bound-I2} and
\eqref{bound-I3}, we proved that
\begin{align}
\label{smooth-Th1}
  \sum_{n = N + 1}^{N + H}
    \e^{-n / N} R_k(n)
  &=
  \frac{\gamma_k^k}{\e} H N^{1 / k}
  +
  \Odip{k}{H N^{1 / k} A \Bigl(N; - \frac12 c_1\Bigr)},
\end{align}
provided that \eqref{H-bound} holds, since the other error terms
are smaller in our range for $H$.
In order to achieve the proof, we have to remove the exponential
factor on the left-hand side, exploiting the fact that, since $H$ is
``small,'' it does not vary too much over the summation range.
We use a sort of bootstrapping argument: since
$\e^{-n / N} \in [\e^{-2}, \e^{-1}]$ for all $n \in [N + 1, N + H]$,
we can easily deduce from \eqref{smooth-Th1} that
\[
  \e^{-2}
  \sum_{n = N + 1}^{N + H} R_k(n)
  \le
  \sum_{n = N + 1}^{N + H}
    \e^{-n / N} R_k(n)
  \ll_k
  H N^{1 / k}.
\]
We can use this weak upper bound to majorise the error term arising
from the development $e^{-x} = 1 + \Odi{x}$ that we need in the
left-hand side of \eqref{smooth-Th1}.
In fact, we have
\begin{align*}
  \sum_{n = N + 1}^{N + H}
    \e^{-n / N} R_k(n)
  &=
  \sum_{n = N + 1}^{N + H}
    \bigl(\e^{-1} + \Odi{(n - N) N^{-1}} \bigr) R_k(n) \\
  &=
  \e^{-1}
  \sum_{n = N + 1}^{N + H} R_k(n)
  +
  \Odip{k}{ H^2 N^{1 / k - 1} }.
\end{align*}
Finally, substituting back into \eqref{smooth-Th1}, we obtain the
required asymptotic formula for $H$ as in the statement of
Theorem~\ref{CGZ-uncond}.

\section{Proof of Theorem \ref{CGZ-cond}}
\label{proof-cond}

The proof of the conditional version of our result is easier since
Lemma~\ref{LP-Lemma-gen} applies to the full unit interval, and this
partially spares us the trouble of dealing with two different ranges.
Recalling \eqref{basic-identity}, we write
\begin{align*}
  \sum_{n = N + 1}^{N + H}
    \e^{-n / N} R_k(n)
  &=
  \gamma_k^{k+1}
  \int_{-1/2}^{1/2} \frac{U(-\alpha, H)}{z^{(k + 1) / k}}
    \, \e(-N \alpha) \, \dx \alpha \\
  &\qquad+
  \int_{-1/2}^{1/2}
    \Bigl(
      \Stilde_k(\alpha)^{k+1}
      -
      \frac{\gamma_k^{k+1}}{z^{(k + 1) / k}}
    \Bigr)
    U(-\alpha, H) \e(-N \alpha) \, \dx \alpha,
\end{align*}
The main term is evaluated as in \S\ref{proof-uncond}, whereas for the
secondary term we have recourse again to \eqref{identity}.
By Lemma~\ref{Laplace-formula} with $X = 1/ 2$ and
Lemma~\ref{mt-evaluation} with $\lambda = 1 / k$, we find that
\begin{align*}
  \int_{-1/2}^{1/2} \frac{U(-\alpha, H)}{z^{(k + 1) / k}}
    \, \e(-N \alpha) \, \dx \alpha
  &=
  \frac1{\gamma_k}
  \sum_{n = N + 1}^{N + H} \e^{-n / N} n^{1 / k}
  +
  \Odip{k}{\frac HN} \\
  &=
  \frac1{\e \gamma_k} H N^{1 / k}
  +
  \Odip{k}{H^2 N^{1/k - 1}}.
\end{align*}
%
\begin{comment}
By essentially the same computations as above, we find that
%
\begin{align*}
  \int_{-1/2}^{1/2} \frac{U(-\alpha, H)}{z^{(k + 1) / k}}
    \, \e(-N \alpha) \, \dx \alpha
  &=
  \frac1{\gamma_k}
  \sum_{n = N + 1}^{N + H} \e^{-n / N} n^{1 / k}
  +
  \Odip{k}{\frac HN} \\
  &=
  \frac1{\gamma_k}
  \sum_{n = N + 1}^{N + H}
    \Bigl( \e^{-1} + \Odi{H N^{-1}} \Bigr) n^{1 / k}
  +
  \Odip{k}{\frac HN} \\
  &=
  \frac1{\e \gamma_k}
  \sum_{n = N + 1}^{N + H} n^{1 / k}
  +
  \Odip{k}{H^2 N^{1/k - 1} + \frac HN} \\
  &=
  \frac1{\e \gamma_k} H N^{1 / k}
  +
  \Odip{k}{N^{1/k} + H^2 N^{1/k - 1}},
\end{align*}
%
since the summand $H N^{-1}$ in the error term is smaller than
$N^{1 / k}$.
\end{comment}
For the secondary term we argue as above, setting
$x = x(\alpha) = \Stilde_k(\alpha)$ and
$y = y(\alpha) = \gamma_k z^{-1/k}$ and using \eqref{identity}.
First we deal with the range $[-1/H, 1/H]$:
\begin{align*}
  \int_{-1 / H}^{1 / H}
  & \Bigl(
      \Stilde_k(\alpha)^{k+1}
      -
      \frac{\gamma_k^{k+1}}{z^{(k + 1) / k}}
    \Bigr)
    U(-\alpha, H) \e(-N \alpha) \, \dx \alpha \\
  &\ll
  H
  \sum_{j = 1}^k
    \int_{-1 / H}^{1 / H}
      \vert x - y \vert^2 \cdot \vert x \vert^{k - j} \cdot
      \vert y \vert^{j - 1} \, \dx \alpha
  +
  H
  \int_{-1 / H}^{1 / H}
    \vert x - y \vert \cdot \vert y \vert^k \, \dx \alpha \\
  &\ll_k
  H
  \sum_{j = 1}^k
    \max_{\alpha} \vert x \vert^{k - j} \cdot
    \max_{\alpha} \vert y \vert^{j - 1}
    \int_{-1 / H}^{1 / H} \vert x - y \vert^2 \, \dx \alpha \\
  &\qquad+
  H
  \Bigl(
  \int_{-1 / H}^{1 / H} \vert x - y \vert^2 \, \dx \alpha
  \int_{-1 / H}^{1 / H} \vert y \vert^{2 k} \, \dx \alpha
  \Bigr)^{1/2} \\
  &\ll_k
  H
  \sum_{j = 1}^k
    N^{(k - j) / k} \cdot N^{(j - 1) / k}
    N^{1 / k} H^{-1} L^2
  +
  H N^{1 / (2 k)} H^{-1 / 2} L
  \Bigl(
  \int_{-1 / H}^{1 / H} \frac{\dx \alpha}{\vert z \vert^2}
  \Bigr)^{1/2} \\
  &\ll_k
  N L^2
  +
  H^{1 / 2} N^{1 / 2 + 1 / (2 k)} L.
\end{align*}
In the remaining range we use a partial-integration argument, in order
to exploit the full force of \eqref{U-bound}.
By Lemma~\ref{LP-Lemma-gen} we have
\[
  F(\xi)
  :=
  \int_0^\xi \bigl\vert \E_k(\alpha) \bigr\vert^2 \, \dx \alpha
  \ll_k
  N^{1 / k} \xi L^2.
\]
Hence
\begin{equation}
\label{avg-E^2-RH}
  \int_{1 / H}^{1 / 2}
    \bigl\vert \E_k(\alpha) \bigr\vert^2 \, \frac{\dx \alpha}{\alpha}
  =
  \Bigl[ \frac{F(\alpha)}{\alpha} \Bigr]_{1 / H}^{1 / 2}
  +
  \int_{1 / H}^{1 / 2}
    \frac{F(\alpha)}{\alpha^2} \, \dx \alpha
  \ll_k
  N^{1 / k} L^3.
\end{equation}
Choosing $x$ and $y$ as above and using \eqref{avg-E^2-RH}, we see
that the contribution from the range $[1/H, 1/2]$ is
\begin{align*}
  \int_{1 / H}^{1 / 2}
  & \Bigl(
      \Stilde_k(\alpha)^{k+1}
      -
      \frac{\gamma_k^{k+1}}{z^{(k + 1) / k}}
    \Bigr)
    U(-\alpha, H) \e(-N \alpha) \, \dx \alpha \\
  &\ll_k
  \sum_{j = 1}^k
    \max_{\alpha} \vert x \vert^{k - j} \cdot
    \max_{\alpha} \vert y \vert^{j - 1}
    \int_{1 / H}^{1 / 2} \vert x - y \vert^2 \, \frac{\dx \alpha}{\alpha}
  +
  \Bigl(
  \int_{1 / H}^{1 / 2} \vert x - y \vert^2 \, \frac{\dx \alpha}\alpha
  \int_{1 / H}^{1 / 2} \vert y \vert^{2 k} \, \frac{\dx \alpha}\alpha
  \Bigr)^{1/2} \\
  &\ll_k
  \sum_{j = 1}^k
    N^{(k - j) / k} \cdot N^{(j - 1) / k}
    N^{1 / k} L^3
  +
  (N^{1 / k} L^3)^{1 / 2} H \\
  &\ll_k
  N L^3
  +
  H N^{1 / (2 k)} L^{3 / 2}.
\end{align*}

\subsection{Completion of the proof}
\label{final-Th2}

Summing up, we proved that
\begin{equation}
\label{smooth-Th2}
  \sum_{n = N + 1}^{N + H}
    \e^{-n / N} R_k(n)
  =
  \frac{\gamma_k^k}{\e} H N^{1 / k}
  +
  \Odip{k}{\Phi_k(N, H)},
\end{equation}
where $\Phi_k(N, H) = N L^3 + H N^{1 / (2 k)} L^{3 / 2} + H^2 N^{1/k - 1} + H^{1 / 2} N^{1 / 2 + 1 / (2 k)} L$.
As in \S\ref{final-Th1} above, we need to remove the exponential
factor, exploiting the fact that, since $H$ is ``small,'' it does not
vary too much over the summation range.
We argue in a slightly different fashion, since we aim at a stronger
error term.
Using the fact that $\e^{-n / N} = \e^{-1} + \Odi{H / N}$, we have
\begin{equation}
\label{Th2-final}
  \sum_{n = N + 1}^{N + H}
    R_k(n)
  =
  \gamma_k^k H N^{1 / k}
  +
  \Odip{k}{\Phi_k(N, H) + \frac HN \sum_{n = N + 1}^{N + H} R_k(n)}.
\end{equation}
The last term is $\ll_k H / N ( H N^{1/k} + \Phi_k(N, H)) \ll_k H^2 N^{1/k - 1} + \Phi_k(N, H) \ll_k \Phi_k(N, H)$ by \eqref{smooth-Th2}, since $H \le N$.
Substituting into \eqref{Th2-final} we find
\[
  \sum_{n = N + 1}^{N + H}
    R_k(n)
  =
  \gamma_k^k H N^{1 / k}
  +
  \Odip{k}{\Phi_k(N, H)}.
\]
This is an asymptotic formula provided that
$H = \infty\bigl(N^{1 - 1 / k} L^3\bigr)$ and $H = \odi{N}$.
Theorem~\ref{CGZ-cond} is fully proved.

\section{Proof of Theorems \ref{CGZ'-uncond} and \ref{CGZ'-cond}}
\label{proof-density-1}

We split the unit interval as in \S\ref{proof-uncond} and proceed
in the same way.
With a similar notation, we find that
\[
  \sum_{n = N + 1}^{N + H}
    \e^{-n / N} R'_k(n)
  =
  \gamma_k^k I'_1 + I'_2 + I'_3,
\]
say, where
\[
  I'_1
  =
  H \e^{-1} + \Odip{k}{H^2 / N + H / B},
\]
by Lemma~\ref{Laplace-formula} and standard estimates.
We also have
\[
  I'_2
  \ll_k
  H A \Bigl(N; - \frac12 c_1\Bigr),
\]
arguing as above by means of identity \eqref{identity} with $k$
replaced by $k - 1$.
For $k \ge 4$ we have
\[
  I'_3
  =
  \int_\C \Stilde_k(\alpha)^k U(-\alpha, H) \e(-N \alpha) \, \dx \alpha
  \ll_k
  \max_{\alpha \in [-1/2, 1/2]}
    \vert \Stilde_k(\alpha) \vert^{k - 4}
  \int_\C \vert \Stilde_k(\alpha) \vert^4 \, \frac{\dx \alpha}\alpha
  \ll_k
  N^{\eps} H / B,
\]
by Lemmas~\ref{Stilde-bound} and~\ref{fourth-power-avg}, which is
$\ll_k H A(N; -c_1 / 2)$, by our choice of $B$.

For $k \in \{2$, $3 \}$ we need the extension to $\Stilde_3$ of Lemma~7
of Tolev \cite{Tolev1992}.
The details of the proof are contained in work in progress by
A.~Languasco and the last Author.
Without this Lemma, for $k = 3$ we need to take $B = N^{5 / 72 + \eps}$
and we obtain a correspondingly weaker result.

\begin{Lemma}[Tolev]
\label{Tolev-Lemma}
Let $k > 1$ and $0 < \tau \le 1 / 2$. Then
\[
  \int_{-\tau}^{\tau} \vert \Stilde_k(\alpha) \vert^2 \, \dx \alpha
  \ll_k
  \bigl(\tau N^{1/k} + N^{2/k - 1}\bigr) (\log N)^3.
\]
\end{Lemma}

A partial-integration argument similar to the ones above then yields
\[
  I'_3
  \ll_k
  \max_{\alpha \in [-1/2, 1/2]}
    \vert \Stilde_3(\alpha) \vert^{k - 2}
  \int_\C \vert \Stilde_k(\alpha) \vert^2 \, \frac{\dx \alpha}\alpha
  \ll_k
  N^{(k - 2) / k}
  \Bigl( N^{1 / k} L^4 + \frac HB N^{(2 - k) / k} L^3 \Bigr),
\]
which is $\odi{H N^{-\eps}}$ if $B = N^{2 \eps}$ as in \eqref{def-B},
for $H$ as in the statement of Theorem~\ref{CGZ'-uncond}.

We omit the details of the proof of Theorem~\ref{CGZ'-cond}.

\smallskip
\noindent
\textbf{Acknowledgement.}
We thank Alessandro Languasco for several conversations on this topic.

\providecommand{\MR}{\relax\ifhmode\unskip\space\fi MR }
\providecommand{\MRhref}[2]{%
  \href{http://www.ams.org/mathscinet-getitem?mr=#1}{#2}
}
\providecommand{\href}[2]{#2}

\bigskip

\begin{tabular}{l}
Dipartimento di Scienze, Matematiche, Fisiche e Informatiche \\
Universit\`a di Parma \\
Parco Area delle Scienze 53/a \\
43124 Parma, Italia \\
email (MC): \texttt{cantarini\_m@libero.it} \\
email (AG): \texttt{a.gambini@unibo.it} \\
email (AZ): \texttt{alessandro.zaccagnini@unipr.it}
\end{tabular}

\end{document}